\documentclass{amsart}

\usepackage{amsmath}
\usepackage{hyperref}
\usepackage{graphicx}
\usepackage{epsfig}
\usepackage{latexsym}
\usepackage{bm}
\usepackage{amssymb}

\newtheorem{theorem}{Theorem}[section]
\newtheorem{defn}[theorem]{Definition}
\newtheorem{lemma}[theorem]{Lemma}
\newtheorem{rem}[theorem]{Remark}

\def\trace{\mathop{\rm trace}}
\def\det{\mathop{\rm det}}

\def\IR{\mathbb{R}}
\def\IR{{\mathbb{R}}}

\def\Span{\mathop{\rm Span}}

\theoremstyle{remark}

\numberwithin{equation}{section}

\usepackage{hyperref}

\begin{document}

\title[FRNK methods]{Functionally fitted\\ Runge-Kutta-Nystr\"{o}m methods}

\author{N. S. Hoang}

\address{Department of Mathematics, University of West Georgia,
Carrollton, GA 30116, USA}

\email{nhoang@westga.edu}

\author{R. B. Sidje}

\address{Department of Mathematics, University of Alabama,
Tuscaloosa, AL 35487, USA}

\email{roger.b.sidje@ua.edu}

\subjclass[2000]{Primary  65R; Secondary }

\date{}

\keywords{second-order IVP, Runge-Kutta-Nystr\"{o}m, functionally fitted, generalized collocation.}

\begin{abstract}
We have shown previously that functionally fitted Runge-Kutta (FRK) methods 
can be studied using a convenient collocation framework.
Here, we extend that framework to
functionally fitted Runge-Kutta-Nystr\"om (FRKN) methods, shedding further light on the fact
that these methods can integrate a second-order differential equation exactly if its solution
is a combination of certain basis functions, and that superconvergence can be obtained 
when the collocation points satisfy some orthogonality conditions.
An analysis of their stability is also conducted.
\end{abstract}

\maketitle

\section{Introduction}

Consider the special second-order differential equation
\begin{equation}
\label{sys}
\begin{split}
y&''(t)=f(t,y(t)),\quad y(t_0)=y_0,\quad y'(t_0)=y'_0,\\
y&:\IR\rightarrow \IR,\quad f:\IR\times\IR\rightarrow \IR,\quad t\in [t_0,\, t_{0}+T],
\end{split}
\end{equation}
where, for simplicity of notation, we 
keep \eqref{sys} in scalar form although our results apply
to a system of ODEs as well. 
We assume that the problem satisfies the necessary conditions to have a unique solution.
Letting $z(t) := y'(t)$, it is well known that
this problem can be converted to a first-order IVP
\begin{equation}
\label{indirectsys}
\begin{pmatrix}
y(t) \\z(t) 
\end{pmatrix}'
=
\begin{pmatrix}
z(t)\\f(t,y(t)) 
\end{pmatrix},\quad
\begin{pmatrix}
y(t_0) \\z(t_0) 
\end{pmatrix}
=
\begin{pmatrix}
y_0\\y_0'
\end{pmatrix},
\end{equation}
so that it can be studied entirely through this commonly called indirect approach.
What is however preferable 
is to address~\eqref{sys} directly without doubling the dimension of the problem
and/or introducing the term $y'(t)$ that is
not present in the initial problem. This is the popular approach taken 
by direct Runge-Kutta-Nystr\"{o}m (RKN) methods.

In~\cite{FRK,SFRK}, we developed a collocation framework
for functionally fitted Runge-Kutta (FRK) methods that can be applied to the
indirect approach~\eqref{indirectsys}. 
We briefly alluded in~\cite[Remark~4.2]{SFRK} to a P-stability result
 adapted from this first-order formulation.
But the prominence of direct RKN methods
motivates recasting our framework to this popular and well-regarded setting. 
Consequently in this paper
we target the special second-order IVP \eqref{sys}
as RKN methods do,
 and we develop functionally fitted Runge-Kutta-Nystr\"{o}m (FRKN) methods in a way that mirrors
what we did for FRK methods.
Although there are similarities with our previous work, this study fills a gap in the literature 
by providing a unifying umbrella for both FRK and FRKN methods.
In a sense, similarities with~\cite{FRK,SFRK} validate our unified framework,
but there are nevertheless differences in the details to make the most of the special form~\eqref{sys}. 
The overlap is kept to a minimum by not repeating
proofs that do not add a distinctive value to our presentation and that can be drawn from our
earlier work, with hints given to the interested reader
on how to recover them.

Among the many studies that have already been devoted to~\eqref{sys} other than
general purpose methods, we can cite exponentially fitted methods \cite{D?Ambrosio,Franco2,Franco&Gomez2013,Pater},
and especially the work of Ozawa~\cite{O2} who studied functionally fitted methods
using Taylor series. Our paper builds on that work but uses collocation techniques.
The organization of the paper is as follows. 
Section~\ref{sec:FRKN} formally defines 
functionally fitted Runge-Kutta-Nystr\"{o}m methods for general basis functions.
Section~\ref{sec:Separable} summarizes properties that arise from basis functions that are separable.
Section~\ref{sec:Accuracy} derives their order of accuracy using our collocation framework.
Section~\ref{sec:Stability} presents new results regarding their stability.
Section~\ref{sec:Examples} is a walk-through example.
Section~\ref{sec:Conclusion} provides some concluding remark.

\section{Functionally fitted RKN methods}
\label{sec:FRKN}

Recall that a conventional $s$-stage RKN method is often defined by its Butcher-tableau as
\begin{displaymath}
\renewcommand{\arraystretch}{1.2}
\begin{array}{c|c}
\bm{c}  &\bm{A} \\ \hline
        &\bm{b}^T\\
	&\bm{d}^T
\end{array},~~
\bm{A} = [a_{ij}] \in \IR^{s\times s},~~ 
\bm{b} = (b_1, \ldots, b_s)^T,~~ 
\bm{d} = (d_1, \ldots, d_s)^T. 
\end{displaymath}
For an explicit RKN method, the matrix $\bm{A}$ is strictly lower triangular and
$c_1=0$. 
Let $\bm{e}=(1,\ldots,1)^T$ of length $s$,
denote ${y}_n$ and ${y}'_n$  the approximations to ${y}(t)$ and ${y}'(t)$ at $t_n$,
then the next iterates
${y}_{n+1}$ and ${y}'_{n+1}$ are computed using
\begin{align}
\label{eq2.1}
y_{n+1} &= y_n + hy'_n+ h^2\bm{b}^Tf(\bm{e}t_n+\bm{c}h,\bm{Y}_n) \in \IR,\\
\label{eq2.2}
y'_{n+1} &= y'_n + h\bm{d}^Tf(\bm{e}t_n+\bm{c}h,\bm{Y}_n) \in \IR,\\
\bm{Y} _n&= \bm{e}y_n+ \bm{c}hy'_n+h^2\bm{A}f(\bm{e}t_n+\bm{c}h,\bm{Y}_n) \in \IR^s,
\end{align}
where 
$\bm{Y}_n=(Y_{n,1},\ldots,Y_{n,s})^T$ is the vector of intermediate stage values and
$f(\bm{e}t_n+\bm{c}h,\bm{Y}_n) = 
(f(t_n+c_1h,Y_{n,1}), \ldots, f(t_n+c_sh,Y_{n,s}))^T$.



Now, assume we are given 
distinct parameters $(c_i)_{i=1}^s$, usually in $[0,1]$, and a
 set of linearly independent basis functions $\{u_k(t)\}_{k=1}^{s}$ that
should not include $\{1,t\}$.
We will augment the set with~$\{1,t\}$ to characterize FRKN methods later
because any problem for which the solution is a linear combination of~$\{1,t\}$ 
is always integrated exactly by FRKN (and RKN) methods, regardless of basis functions.
In general, basis functions are typically chosen
to exploit any information on the solution that general purpose methods do not.
The linear independence is to avoid undue redundancy in the set by not including
a function that is merely a linear combination of the others.
Functionally fitted (or generalized collocation) RKN
methods are defined to solve \eqref{sys} exactly
if its solution is a linear combination of the chosen functions (Ozawa~\cite{O2}).

\begin{defn}[Functionally fitted RKN method]
\label{ffRKdef}
An $s$-stage RKN method is an FRKN 
(or a generalized collocation RKN) method with respect 
to the basis functions $\{u_k(t)\}_{k=1}^{s}$ if the 
following relations are satisfied for all $k=1,\ldots,s$:
\begin{equation}
\label{fFRKN}
\begin{split}
u_k(t+h)&=u_k(t)+hu_k'(t)+h^2\bm{b}(t,h)^Tu_k''(\bm{e}t+\bm{c}h),\\
u'_k(t+h)&=u'_k(t)+h\bm{d}(t,h)^Tu_k''(\bm{e}t+\bm{c}h),\\
u_k(\bm{e}t+\bm{c}h)&=\bm{e}u_k(t)+h\bm{c}u'_k(t)+h^2\bm{A}(t,h)u_k''(\bm{e}t+\bm{c}h).
\end{split}
\end{equation}
\end{defn}

The underlying coefficients $\bm{A}(t,h)$, $\bm{b}(t,h)$ and $\bm{d}(t,h)$ are  
often determined by solving the linear system of algebraic equations
arising from \eqref{fFRKN}.
When the basis functions are the monomials, classical algebraic collocation
RKN methods \cite{cong91} are recovered.
 It is clear that a solution to \eqref{fFRKN} may or may not exist depending
 on the choice of $\{u_k(t)\}_{k=1}^{s}$. 
To guarantee the existence of a solution we require that the basis functions 
satisfy the {\it collocation condition} defined below. 

\begin{defn}[Collocation condition for FRKN methods]
\label{defn2.2}
A set of sufficiently smooth functions $\{u_1(t),\ldots,$ $u_s(t)\}$ is said 
to satisfy the collocation condition  for FRKN methods if the following matrices
\begin{align*}
\bm{E}(t,h)&=\bigg{(}u_{1}(\bm{e}t+\bm{c}h)-u_1(\bm{e}t)-\bm{c}hu'_1(t)
,\ldots,u_{s}(\bm{e}t+\bm{c}h)-u_s(\bm{e}t)-\bm{c}hu'_s(t)\bigg{)},\\
\bm{F}(t,h)&=\bigg{(}u_{1}''(\bm{e}t+\bm{c}h),u_{2}''(\bm{e}t+\bm{c}h),\ldots,u_{s}''(\bm{e}t+\bm{c}h)\bigg{)},
\end{align*}
are both nonsingular almost everywhere with respect to $h$ on $[0,T]$ for any given $t$.
\end{defn}
%

\begin{rem}\rm
Note that the last equation in \eqref{fFRKN} can be rewritten as 
$$
\bm{E}(t,h)=h^2\bm{A}(t,h)\bm{F}(t,h).
$$
Hence the existence of the matrix $\bm{A}(t,h)$  is guaranteed if $\bm{F}(t,h)$ is nonsingular,
as required in Definition~\ref{defn2.2}. We shall see later that by also requiring $\bm{E}(t, h)$ to be nonsingular we ensure 
the existence of the so-called {\em collocation solution}.
Requiring both $\bm{E}(t, h)$ and $\bm{F}(t, h)$  nonsingular
makes $\bm{A}(t,h)$  nonsingular as well, leading to FRKN methods that are
implicit. By allowing $\bm{A}(t,h)$ to be strictly lower triangulary (thus singular), 
Franco and G{\`o}mez~\cite{Franco&Gomez2013} obtained
 explicit  exponentially fitted RKN methods.
\end{rem}

\begin{rem}\rm
In \cite{O2}, Ozawa used Taylor series expansions to prove that 
when given any set of functions $\{\varphi_k(t)\}_{k=1}^s$,
if the Wronskian $W(\varphi_1,\varphi_2,\ldots,\varphi_s)(t) \ne 0$, then the matrix $\bm{\Phi}(t,h)
=\big{(}\varphi_{1}(\bm{e}t+\bm{c}h),\varphi_{2}(\bm{e}t+\bm{c}h),\ldots,\varphi_{s}(\bm{e}t+\bm{c}h)\big{)}$
is nonsingular for small $h>0$.
We can therefore conclude that 
$\bm{F}(t,h)$ is nonsingular for $h$ sufficiently small if
$W(u_1'',u_2'',\ldots,u_s'')(t) \ne 0,\,
\forall t\in[t_0,\,t_0+T]$. Moreover, using the same Taylor series approach, 
we can see that
$$
\bm{E}(t,h)=
\bigg{(}\frac{(\bm{c}h)^2}{2!},\frac{(\bm{c}h)^3}{3!},\ldots,\frac{(\bm{c}h)^{s+1}}{(s+1)!}\bigg{)}
W(u_1'',u_2'',\ldots,u_s'')(t)+O(h^{s+2}),
$$
so that $W(u_1'',u_2'',\ldots,u_s'')(t) \ne 0,
\forall t\in[t_0,\,t_0+T]$ is a sufficient (but not necessary)
condition for both $\det \bm{E}(t,h)\ne 0$ and $\det \bm{F}(t,h)\ne 0$ when $h$ is sufficiently small. 
Thus, our condition is less restrictive than the one in \cite{O2}.
Indeed it is easily verified that the set 
$\{\sin(\omega t),\sin(2\omega t)\}$ 
does not satisfy Ozawa's condition but satisfy 
our collocation condition.
\end{rem}

The following is a direct result of the collocation condition in Definition~\ref{defn2.2}. 

\begin{theorem}
\label{theorem2.3}
The coefficients of a FRKN method based on a set of functions 
that satisfy the collocation condition are uniquely determined 
almost everywhere with respect to $h$ on the interval of integration.
\end{theorem}

As we did with FRK methods in~\cite{FRK,SFRK}, 
our overarching contribution is to
study the order of accuracy of FRKN methods
without using order conditions. We do so by establishing the existence of a fundamental function 
that we call the {\em collocation solution}. We also use this function later to
analyze the stability.

Choose $(c_i)_{i=1}^s$ and consider a FRKN method
$(\bm{c}$, $\bm{A}(t,h)$, $\bm{b}(t,h)$, $\bm{d}(t,h))$ based on a set of given basis functions
$\{u_k\}_{k=1}^s$. Let 
\begin{align*}
\bm{H}:&=\text{Span}\{1,t,u_1(t),\ldots,u_s(t)\}\\
&=\bigg\{v(t)\bigg| 
v(t)= \alpha_{0}+a_0t+\sum_{i=1}^sa_iu_i(t);\, \alpha_0,a_0,\ldots,a_s\in \mathbb{R}\bigg\}.
\end{align*}
We call $u(t)$ the {\em collocation solution}
if it is an element of $\bm{H}$ that satisfies equation \eqref{sys} at the collocation points $(hc_i)_{i=1}^s$, that is,
\begin{equation}
\label{evenpoly}
u(t_0) = y_0,\quad u'(t_0) = y'_0,\quad
u''(t_0+c_ih) = f(t_0+c_ih,u(t_0+c_ih)),
\end{equation}
for $i = 1, \ldots, s$.
As it is well known, the {\em collocation method} consists in taking the numerical solution after one step 
as 
\begin{equation}
\label{colmethod}
y_1 = u(t_0 + h), \quad y_1' = u'(t_0+h). 
\end{equation}
Since $u(t)$ is only defined implicitly, we must first guarantee its existence in our context as well.
 
\begin{lemma}\label{newinter}
Suppose that we are given $s+2$ values $y_0, y'_0,y_1, \ldots, y_s$  
and the pair $(t_0,h)$ is such that $\bm{E}(t_0,h)$ is nonsingular, 
then there exists an interpolation function $\varphi \in\bm{H}$ such 
that $\varphi(t_0) = y_0$, $\varphi'(t_0) = y'_0$, and $\varphi(t_0+c_ih)=y_i$, $i=1,\ldots,s$.
\end{lemma}

\begin{proof}Any function $\varphi \in \bm{H}$ can 
be represented in the form
$$
\varphi(t)= \alpha_0+a_0t+\sum_{i=1}^{s}a_iu_i(t).
$$
Letting $t_i:=t_0+c_ih$, $i=1,\ldots,s$,
the interpolation criteria can be stated therefore as
\begin{equation}
\label{eqm3.0}
\begin{pmatrix}
1&	t_0& u_{1}(t_0)  & u_{2}(t_0)  & \cdots & u_{s}(t_0)\\
0&	1& u'_{1}(t_0)  & u'_{2}(t_0)  & \cdots & u'_{s}(t_0)\\
1&	t_1& u_{1}(t_1)  & u_{2}(t_1)  & \cdots & u_{s}(t_1)\\
\vdots &\vdots    & \vdots    & \ddots & \vdots\\
1&	t_s& u_{1}(t_s)  & u_{2}(t_s)  & \cdots & u_{s}(t_s)
\end{pmatrix}
\begin{pmatrix}
\alpha_0\\
a_0\\
a_1\\
\vdots\\
a_s
\end{pmatrix}
=
\begin{pmatrix}
y_0\\
y'_0\\
y_1\\
\vdots\\
y_s
\end{pmatrix}.
\end{equation}
For this equation to have a unique solution the matrix in 
the left-hand side must be nonsingular. Subtracting its first
row from the third to the last row, we see that its determinant is
\begin{equation}
\label{eqm3.1}
\begin{vmatrix}
1   &u'_{1}(t_0)            &u'_{2}(t_0)            & \cdots &u'_{s}(t_0)\\
c_1h&u_{1}(t_1)-u_{1}(t_0)  &u_{2}(t_1)-u_{2}(t_0)  & \cdots &u_{s}(t_1)-u_{s}(t_0)\\
c_2h&u_{1}(t_2)-u_{1}(t_0)  &u_{2}(t_2)-u_{2}(t_0)  & \cdots &u_{s}(t_2)-u_{s}(t_0)\\
\vdots &\vdots    & \vdots    & \ddots & \vdots\\
c_sh&u_{1}(t_s)-u_{1}(t_0)  &u_{2}(t_s)-u_{2}(t_0)  & \cdots &u_{s}(t_s)-u_{s}(t_0)
\end{vmatrix}.
\end{equation}
Multiplying the first row in the determinant in \eqref{eqm3.1} by 
$c_ih$ and subtracting it from the $(i+1)$-th row for $i=1,\ldots,s$, 
one concludes that this determinant is
\begin{equation*}
\begin{vmatrix}
u_{1}(t_1)-u_{1}(t_0)- c_1hu'_{1}(t_0)  & \cdots &u_{s}(t_1)-u_{s}(t_0) - c_1hu'_{s}(t_0)\\
u_{1}(t_2)-u_{1}(t_0)- c_2hu'_{1}(t_0)  & \cdots &u_{s}(t_2)-u_{s}(t_0) - c_2hu'_{s}(t_0)\\
\vdots     & \ddots & \vdots\\
u_{1}(t_s)-u_{1}(t_0)- c_shu'_{1}(t_0)  & \cdots &u_{s}(t_s)-u_{s}(t_0) - c_shu'_{s}(t_0)
\end{vmatrix}
= \det\bm{E}(t_0,h).
\end{equation*}
Therefore, the solution to \eqref{eqm3.0} is uniquely determined since
$\bm{E}(t_0,h)$ is assumed nonsingular under the collocation condition.
\end{proof}

\begin{theorem}
\label{theorem3.10}
The collocation method defined by \eqref{evenpoly} and \eqref{colmethod} is equivalent to the $s$-stage
FRKN method with coefficients
$(\bm{c}$, $\bm{A}(t,h)$, $\bm{b}(t,h)$, $\bm{d}(t,h))$.
\end{theorem}

\begin{proof}
Consider the equations we have to solve in an $s$-stage FRKN method
\begin{equation}
\label{eqm4.1}
Y_i = y_0 +hc_iy'_0+ h^2 \sum_{j=1}^{s}a_{ij}f(t_0+c_jh,Y_j),\quad i=1,\ldots,s.
\end{equation} 
Let $(\overline{Y}_1,\ldots,\overline{Y}_s)^T$ be the unique solution to the system of equations \eqref{eqm4.1}.
(Using fixed-point iterations one can show that this solution does indeed exist for a sufficiently small $h>0$.) 
Lemma~\ref{newinter} ensures that there is an interpolation function $\varphi(t)\in \bm{H}$ such that
$\varphi'(t_0)=y'_0$,  
$\varphi(t_0)=y_0$, and $\varphi(t_0+c_ih)=\overline{Y_i}$, $i=1,\ldots,s$. 
Hence $\varphi(t)$ satisfies
\begin{equation}
\label{cuoi10}
\varphi(\bm{e}t_0+\bm{c}h)=\bm{e}\varphi(t_0) +\bm{c}h\varphi'(t_0) + h^2 \bm{A}(t_0,h)f(\bm{e}t_0+\bm{c}h,\varphi(\bm{e}t_0+\bm{c}h)).
\end{equation}
Since $\varphi(t)\in\bm{H}=\Span\{1,t,u_1(t),\ldots,u_s(t)\}$, we can write
\begin{align*}
\varphi(t)=\alpha_0 + a_0t + \sum_{k=1}^sa_ku_k(t),\quad
\varphi'(t)=a_0 + \sum_{k=1}^sa_ku_k'(t),\quad
\varphi''(t)=\sum_{k=1}^sa_ku_k''(t),
\end{align*}
and since the definition of a FRKN method \eqref{fFRKN} means that
\begin{align*}
u_k(\bm{e}t_0+\bm{c}h)=\bm{e}u_k(t_0)+\bm{c}hu'_k(t_0)+h^2\bm{A}(t_0,h)u_k''(\bm{e}t_0+\bm{c}h),\quad k=1,\ldots,s,
\end{align*}
we can obtain
\begin{align}
\label{cuoi2}
\varphi(\bm{e}t_0+\bm{c}h)=\bm{e}\varphi(t_0)+\bm{c}h\varphi'(t_0)+ h^2\bm{A}(t_0,h)\varphi''(\bm{e}t_0+\bm{c}h).
\end{align}
Since $\bm{A}(t_0,h)$ is nonsingular, 
equations \eqref{cuoi10} and \eqref{cuoi2} imply that
$$
\varphi''(t_0+c_ih)=f(t_0+c_ih,\varphi(t_0+c_ih)),\qquad i=1,\ldots,s.
$$
Therefore, if we choose $u(t)=\varphi(t)$, then $u(t)$ 
satisfies equality \eqref{evenpoly}, which proves the theorem. 
\end{proof}

\begin{rem}\rm 
It follows from the proof of Theorem~\ref{theorem3.10} that 
the existence of a solution to the system \eqref{eqm4.1} implies 
the existence of the collocation solution $u(t)$. 
In other words, the collocation solution $u(t)$ exists and is unique for a given $h>0$ 
whenever the associated FRKN method is applicable. 
For a sufficiently small $h>0$, we pointed out in the proof of Theorem~\ref{theorem3.10} that 
the solution to the system~\eqref{eqm4.1} exists and is unique, and so we can conclude 
that the collocation solution exists and is unique as well. 
Moreover, as an element of $\bm{H}$, it is
a linear combination of the given basis functions $\{u_k(t)\}_{k=1}^{s}$ 
augmented with $\{1,t\}$.
\end{rem}

\section{Separable basis}
\label{sec:Separable}



It is clear from our discussion so far that the analysis of FRKN methods 
is complicated by the fact that they have variable coefficients 
that depend on $t$, $h$ and the basis functions in a non-trivial manner. 
In \cite{FRK} and \cite{SFRK}, we introduced the class
of separable methods to overcome the difficult for FRK methods. Here we extend this notion to FRKN methods.


\begin{defn}[Separable basis]
A set of linearly independent functions $\{u_k\}_{k=1}^s$ is said to be 
a separable basis for FRKN methods if $\bm{u}(t) := (1, t, u_1(t), \ldots, u_s(t))^T$ satisfies
\begin{equation}
\label{xyplusmatrix}
\bm{u}(t+h) = \mbox{\boldmath$\mathcal F$}(h)\bm{u}(t) = \mbox{\boldmath$\mathcal F$}(t)\bm{u}(h), \quad \forall t,h \in \mathbb{R},
\end{equation}
where $\mbox{\boldmath$\mathcal F$}:\mathbb{R}\to\mathbb{R}^{(s+2)\times (s+2)}$ 
is a suitable matrix function. 
\end{defn}

FRKN methods corresponding to separable bases are called 
separable methods. The following results characterize these methods further. 
Theorem~\ref{thm-Smatrix} gives an indication 
as to what type of basis functions
can be separable. Theorem~\ref{thm-sepbasis-inclusion}
provides an effective procedure for identifying and constructing separable methods.
Theorems~\ref{seponly} and~\ref{indepentcoeff} show that coefficients based on separable functions are 
time-independent (in the sense that they only depend on the current stepsize)
and that no other class of basis functions
can lead to time-independent coefficients.
Theorem~\ref{sepcolcond} indicates that separable methods 
always satisfy the collocation condition.
Omitted proofs are similar to that  found in the indicated references.

\begin{theorem} \cite[Theorem 3.2]{SFRK}
\label{thm-Smatrix}
If $\bm{u}(t) = (1, t, u_1(t), \ldots, u_s(t))^T$ is separable 
according to \eqref{xyplusmatrix}, then there exists a 
constant and nonderogatory matrix $\bm{S} = \mbox{\boldmath$\mathcal F$}'(0)$
such that $\mbox{\boldmath$\mathcal F$}(t)= e^{\bm{S}t}$. Hence from \eqref{xyplusmatrix},
$\bm{u}(t) = \mbox{\boldmath$\mathcal F$}(t)\bm{u}(0) =  e^{\bm{S}t}\bm{u}_0$.
\end{theorem}

\begin{theorem} \cite[Theorem 3.3]{SFRK}
\label{thm-sepbasis-inclusion}
Given a set of linearly independent functions $\{u_k(t)\}_{k=1}^{s}$,
let $\bm{H}' = \Span\{u_1'(t), \ldots, u_s'(t)\}$.
Then $\bm{u}(t) = (1,t, u_1(t), \ldots, u_s(t))^T$ 
is separable if and only if $\bm{H}' \subset \bm{H}$.
\end{theorem}


\begin{theorem} \cite[Theorem 3.4]{SFRK}
\label{seponly}
The coefficients of an $s$-stage FRKN method are time-independent if and only if
its associated basis is separable. 
\end{theorem}

\begin{theorem} \cite[Theorem 3.5]{SFRK}
\label{sepcolcond}
A separable basis satisfies the collocation condition at any $t$.
\end{theorem}

\begin{theorem} \cite[Theorem 3.1]{SFRK}
\label{indepentcoeff}
If $\{u_k\}_{k=1}^s$ is separable,  
then the coefficients of the corresponding $s$-stage FRKN method exist
almost everywhere with respect to $h$ and are independent of $t$.
These coefficients are $(\bm{c}, \bm{A}(0,h), \bm{b}(0,h), \bm{d}(0,h))$.
\end{theorem}

\begin{proof}
We give the proof of this theorem to emphasize results that we use later.
Since $\{u_k\}_{k=1}^s$ is separable, it follows from Theorem~\ref{sepcolcond} that 
the collocation condition is satisfied almost everywhere with respect to $h$ for any given $t$. 
This and Theorem~\ref{theorem2.3} imply that the coefficients of the corresponding $s$-stage FRKN method 
exist almost everywhere with respect to $h$ for any given $t$. 

To prove that these coefficients are independent of $t$, we show that the following equations
hold for all $t$ and $1 \le k \le s$:
\begin{equation}
\label{eqm10.1}
\begin{split}
u_k(t+h)&=u_k(t)+hu_k'(t)+h^2\bm{b}(0,h)^Tu_k''(\bm{e}t+\bm{c}h),\\
u'_k(t+h)&=u'_k(t)+h\bm{d}(0,h)^Tu_k''(\bm{e}t+\bm{c}h),\\
u_k(\bm{e}t+\bm{c}h)&=\bm{e}u_k(t)+h\bm{c}u'_k(t)+h^2\bm{A}(0,h)u_k''(\bm{e}t+\bm{c}h).
\end{split}
\end{equation}

We only prove the second equality since 
the other equalities can be obtained similarly. 
Proving the second equality in \eqref{eqm10.1} is equivalent to proving
\begin{equation}
\label{eqm10.2}
\bm{u}'(t+h) = \bm{u}'(t) + h\begin{bmatrix}
\bm{u}''(t+c_1h) & \bm{u}''(t+c_2h)&\cdots & \bm{u}''(t+c_sh)
\end{bmatrix}
\bm{d}(0,h).
\end{equation}
Consider first the case $t=0$, equation \eqref{eqm10.2} becomes
\begin{equation}
\label{eqm10.3}
\bm{u}'(h) = \bm{u}'(0) + h\begin{bmatrix}
\bm{u}''(c_1h) & \bm{u}''(c_2h)&\cdots & \bm{u}''(c_sh)
\end{bmatrix}
\bm{d}(0,h).
\end{equation}
This equation holds by Definition \ref{ffRKdef}. 
Consider now the case $t\ne0$,
since $\bm{u}(t)=e^{\bm{S}t}\bm{u}_0$
as shown in Theorem~\ref{thm-Smatrix}, one gets
\begin{equation}
\label{eqm10.4}
\bm{u}'(t+h)= e^{\bm{S}t}\bm{u}'(h),\quad \bm{u}''(t+c_ih)= e^{\bm{S}t}\bm{u}''(c_ih),\quad i=1,...,s.
\end{equation}
Equation \eqref{eqm10.2} follows from \eqref{eqm10.3} and \eqref{eqm10.4}. 
\end{proof}

\begin{rem}\rm
Theorem~\ref{thm-Smatrix} states
that for any separable system of functions $\{u_i(t)\}_{i=1}^s$, the vector function 
$\bm{u}(t)=(1,t,u_1(t), \ldots, u_s(t))^T$ is of the form $\bm{u}(t)=e^{\bm{S}t}\bm{u}_0$
with a nonderogatory, singular matrix $\bm{S}$ and a suitable $\bm{u}_0$, 
and this characterizes completely what a separable system of basis functions is. 
As explained in~\cite{SFRK}, the constrain that $\bm{S}$ is
singular amounts to having the constant function $1$ in $\bm{u}(t)$, 
and the constrain that it is nonderogatory amounts to having functions  
that are linearly independent.
Also note from Theorem~\ref{thm-sepbasis-inclusion} that if 
$\bm{u}(t)=(1,t,u_1(t), \ldots, u_s(t))^T$  is separable, then $\Span\{u_1''(t), \ldots, u_s''(t)\} \subset \bm{H}$. 
\end{rem}


\begin{rem}\label{separable-examples}\rm
As Theorem~\ref{indepentcoeff} claims the coefficients of an $s$-stage FRKN 
method depend only on $h$ when $\{u_k\}_{k=1}^s$ satisfy \eqref{xyplusmatrix}, i.e, the set $\{u_k\}_{k=1}^s$ is 
a separable basis. 
Examples of separable bases include:
\begin{enumerate}
\item $\{u_k(x)\}_{k=1}^{s} = 
\{x^2, \ldots, x^{s+1}\}$.
\item $\{u_k(x)\}_{k=1}^{2n} = 
\{\sin(\omega_1 x), \cos(\omega_1 x), \ldots, \sin(\omega_n x), \cos(\omega_n x)\}$.
\item $\{u_k(x)\}_{k=1}^{2m+n-1} = \{\sin(\omega x), 
\cos(\omega x), \ldots, \sin(m\omega x),\cos(m\omega x)\}
\cup \{x^2, \ldots, x^{n}\}$.
\item $\{u_k(x)\}_{k=1}^{2(n+1)} = \{\sin(\omega x), \cos(\omega x), 
\ldots, x^n\sin(\omega x), x^n\cos(\omega x)\}$.
\item $\{u_k(x)\}_{k=1}^{n+2m+1} = \{x^2, \ldots, x^{n}, 
\exp(\pm w x),x\exp(\pm w x), \ldots, x^m\exp(\pm w x)\}$.
\end{enumerate}
\end{rem}

Therefore, algebraic polynomials, exponentials, sine-cosine and hyperbolic sine-cosine 
functions, and various combinations are in this class. 
Combining functions of different type is also called {\em mixed collocation} as done by
Coleman and Duxbury ~\cite{Coleman} who used
the particular set of basis functions 
$\{\sin(\omega x), \cos(\omega x)\} \cup \{1, x, \ldots, x^{s-1}\}$.
When the set $\{u_k\}_{k=1}^s$ is 
a separable basis, the coefficients of the corresponding FRKN method are independent of $t$ and 
we will simply set
\begin{equation}
\label{eq:coef}
\bm{A}(h):=\bm{A}(0,h),\quad \bm{b}(h):=\bm{b}(0,h),\quad\bm{d}(h):= \bm{d}(0,h).
\end{equation}
It is also worth noting that the coefficients in \eqref{eq:coef} can be frozen or generated
with a value of $h$ different from the actual step size, yielding general purpose
methods.  But our analysis here focuses only on functionally fitted methods.

\section{Order of Accuracy}
\label{sec:Accuracy}

\subsection{Order}

For the order of accuracy we use the following definition:
\begin{defn}
\label{defnorder}
Let $y(t)$ be the exact solution of \eqref{sys} with $y(t_{n})= y_n$, $y'(t_n) = y'_n$. 
Let $y_{n+1}$, $y'_{n+1}$, and $\bm{Y}_n$ be the approximate values of $y(t_{n+1})$, $y'(t_{n+1})$, and $y(\bm{e}t_n+\bm{c}h)$ 
obtained by an $s$-stage FRKN method. 
Assume that 
\begin{align}
\label{defnorderc}
y(t_{n+1}) - y_{n+1} &= O(h^{p_1 + 1}),\\
\label{defnordera}
 y'(t_{n+1}) - y'_{n+1} &= O(h^{p_2 + 1}),\\
\label{defnorderb}
\bm{Y}_n - y(\bm{e}t_n + \bm{c}h) &= O(h^{p_3+1}).
\end{align}
Then the (global) order of accuracy $p$ and the (global) stage order $r$ of the FRKN method are respectively defined by 
$p=\min \{p_1,p_2\}$ and $r=\min\{p_1,p_2,p_3\}$.
\end{defn}

\begin{rem}
\label{defnaccuracy}\rm 
Definition~\ref{defnorder} is slightly different from the definition used
in \cite{cong91} for the order  of accuracy of RKN methods. 
In \cite{cong91}, equations \eqref{defnorderc} and \eqref{defnordera} 
are also used but equation \eqref{defnorderb} is replaced by 
\begin{equation}
\label{congdef}
y(\bm{e}t_n + \bm{c}h) - \bm{e}y_n - \bm{c}hy_n' - h^2\bm{A}f(\bm{e}t_n + \bm{c}h,y(\bm{e}t_n + \bm{c}h)) = O(h^{p_3+1}). 
\end{equation}
Compared to the definition in \cite{cong91}, Definition~\ref{defnorder} is more direct because
it directly uses the accuracy of stage values $\bm{Y_n}$, whereas the order of accuracy  of  stage values $\bm{Y}_n$ in \cite{cong91} can only be achieved implicitly from equation \eqref{congdef}. 
In either case, the {\em local} stage order is $p_3+1$.

The order of accuracy of FRKN methods was also defined by Ozawa in \cite{O2}, where
he used $r=p_3$ as his definition of stage order, c.f.~\cite[Eq.~(12)]{O2}, instead of $r=\min\{p_1,p_2,p_3\}$, and so his definition of stage order did not include equation~\eqref{defnordera}.
With his definition, Ozawa proved in \cite[Theorem 2]{O2} that 
the (global) stage order of
an $s$-stage FRKN method is $s+1$ and the (global) step order is $s$ for any $(c_i)_{i=1}^s$. 
But his definition has the drawback of not conforming with the definition of the stage order 
of Runge-Kutta methods.
\end{rem}

\begin{rem}
\label{smooth1}\rm 
Suppose that $f(t)\in C^{m+n}[0,T]$ 
and that $f(c_ih)=0,i=1,\ldots,n$. Then the function $\zeta(t):=\frac{f(t)}{(t-c_1h)\cdots(t-c_nh)}$ 
has $(c_ih)_{i=1}^n$ as removable singularities. 
Thus, $\zeta(t)$ can be extended to be an element of $C^{m}[0,T]$
as shown in~\cite[Remark~3.1]{FRK}.
\end{rem}

\begin{theorem}
\label{thm3.2}
The order of an $s$-stage FRKN method is $p=s$ and the stage order is $r=s$.  
\end{theorem}

\begin{proof}
Let us revisit Ozawa's result~\cite[Theorem~2]{O2} using our collocation framework.
Assume without loss of generality that $t_0=0$. 
Let $u(t)$ be the collocation solution corresponding to an $s$-stage FRKN method 
$(\bm{c}, \bm{A}, \bm{b}, \bm{d})$. 
Thus, $u(t)$ satisfies \eqref{evenpoly} and the error function $u''(t)-f(t,u(t))$ 
is zero at $t=c_ih, i=1,\ldots,s$. Let
\begin{equation}
\label{eq3.1.0}
g(t):=\frac{u''(t)-f(t,u(t))}{\prod_{i=1}^s(t-c_ih)},\qquad t\not=c_ih,\quad i=1,\ldots,s.
\end{equation}
This function can be extended over the interval $[0,T]$ provided that $f$ is sufficiently smooth 
as mentioned in Remark~\ref{smooth1}. We can equivalently write equation \eqref{eq3.1.0} as
\begin{equation}
\label{eq3.1}
u''(t)-f(t,u(t))=g(t)\prod_{i=1}^s(t-c_ih).
\end{equation}
Let 
\begin{equation}
\label{eqError}
R(t):=u(t)-y(t),
\end{equation}
where $y(t)$ is the solution to problem \eqref{sys}, i.e., 
$$
y''(t)=f(t,y(t)),\qquad y(0)=y_0,\qquad y'(0)=y'_0.
$$  
This and equation \eqref{eq3.1} imply
\begin{align}
\label{supperconvergence1}
R''(t)= u''(t)-y''(t) = [f(t,u(t))-f(t,y(t))]+g(t)\prod_{i=1}^s(t-c_ih).
\end{align}
Define
\begin{align}
\label{supperconvergence2}
L(t)=\left\{
\begin{matrix}
\frac{f(t,u(t))-f(t,y(t))}{u(t)-y(t)} &\text{if}\quad u(t)\neq y(t),\\*[3mm]
\frac{\partial f}{\partial y}(t,y(t)) &\text{if}\quad u(t)= y(t).
\end{matrix}\right.
\end{align}
With $f(t,y)$ assumed Lipschitz continuous in $y$ there exists a constant $L>0$ such that
$$
 |L(t)| \leq L,\quad t\in [0,h].
$$
From \eqref{supperconvergence2} one gets
$$
f(t,u(t))-f(t,y(t)) =  L(t)[u(t)-y(t)] = L(t)R(t).
$$
This and \eqref{supperconvergence1} imply
$$
R''(t) = L(t)R(t) + g(t)\prod_{i=1}^s(t-c_ih),\quad R(0)=R'(0)=0.
$$
This equation can be written as
\begin{equation}
\label{eqm10.6}
\begin{pmatrix}
R(t)\\R'(t)
\end{pmatrix}' = 
\begin{pmatrix}
0&1\\ 
L(t)& 0
\end{pmatrix}\begin{pmatrix}
R(t)\\R'(t)
\end{pmatrix} + 
\begin{pmatrix}
0\\
g(t)\prod_{i=1}^s(t-c_ih)
\end{pmatrix}
,\quad 
\begin{pmatrix}
R(0)\\R'(0)
\end{pmatrix}=\bm{0}.
\end{equation}
From the theory of ordinary differential equations, the solution to equation~\eqref{eqm10.6} 
is 
\begin{equation}
\label{eqRode}
\begin{pmatrix}
R(t)\\R'(t)
\end{pmatrix} = \bm{X}(t) \int_0^t \bm{X}^{-1}(\xi) \begin{pmatrix}
0\\
g(\xi)\prod_{i=1}^s(\xi -c_ih)
\end{pmatrix} \,d\xi,\qquad t\ge 0,
\end{equation}
where $\bm{X}(t)$ is the fundamental matrix solution to the homogeneous problem corresponding to equation~\eqref{eqm10.6}, i.e., $\bm{X}(t)$ solves the problem
\begin{equation*}
\bm{X}'(t) = \begin{pmatrix}
0&1\\ 
L(t)& 0
\end{pmatrix}\bm{X}(t),\qquad \bm{X}(0) = 
\begin{pmatrix}
1 & 0\\
0 & 1
\end{pmatrix}. 
\end{equation*}
Let 
\begin{equation}
\label{eq3.14}
\bm{X}(t) = \begin{pmatrix}a_{11}(t) & a_{12}(t)\\ a_{21}(t) & a_{22}(t)\end{pmatrix},\qquad
\bm{X}^{-1}(t) = \begin{pmatrix}b_{11}(t) & b_{12}(t)\\ b_{21}(t) & b_{22}(t)\end{pmatrix}.
\end{equation}
From \eqref{eqRode} and \eqref{eq3.14} one gets
\begin{align*}
R'(t)=& \int_0^t [a_{21}(t)b_{12}(\xi)+ a_{22}(t)b_{22}(\xi)]g(\xi)\prod_{i=1}^s(\xi-c_ih)\,d\xi.
\end{align*}
Therefore, 
\begin{equation}
\label{eq3.16}
R'(th) = h^{s+1}\int_0^t [a_{21}(th)b_{12}(\xi h)+ a_{22}(th)b_{22}(\xi h)]g(\xi h)\prod_{i=1}^s(\xi-c_i)\,d\xi.
\end{equation}
Recall~\eqref{eqError} and the collocation method~\eqref{colmethod}, equation~\eqref{eq3.16} 
implies that
\begin{equation}
\label{eqEp}
y'_{1}-y'(h)=R'(h) = O(h^{s+1}),
\end{equation}

From the Fundamental Theorem of Calculus and equation \eqref{eq3.16} one gets
\begin{align}
\label{eq3.15}
R(\omega h) &= R(0) + \int_0^\omega R'(th) h\,dt \\
      &= h^{s+2}\int_0^\omega\int_0^t [a_{21}(th)b_{12}(\xi h)+ a_{22}(th)b_{22}(\xi h)]g(\xi h)\prod_{i=1}^s(\xi-c_i)\,d\xi dt. 
      \nonumber
\end{align}
Letting $\omega = c_i$ and $\omega = 1$ in equation \eqref{eq3.15} one obtains the following relations
\begin{align}
Y_i-y(c_ih)&=u(c_ih)-y(c_ih)=R(c_ih)=O(h^{s+2}),\quad i=1,...,s, \label{eqEstage}\\
y_{1}-y(h)&=u(h)-y(h)=R(h)=O(h^{s+2}). \label{eqE}
\end{align}

It follows from Definition \ref{defnorder} and equations \eqref{eqEp}, \eqref{eqE}, and \eqref{eqEstage} 
that both the accuracy order and stage order are equal to $s$. 
Theorem~\ref{thm3.2} is proved.
\end{proof}

\begin{rem}\rm 
It follows from equations \eqref{eqEp}, \eqref{eqE}, and \eqref{eqEstage} that FRKN methods yield 
$y_{n+1}$ and $\bm{Y}_n$ with a local accuracy order of $s+2$ and yield $y'_{n+1}$ with a local accuracy order of $s+1$.  
This raises the question whether the global accuracy order can be improved to $s+1$ instead of $s$ if one can obtain $y'_{n+1}$ with a local accuracy order of $s+2$. The answer is yes. Indeed, in our numerical experiments we have been able to increase the global accuracy order to $s+1$ for any $(c_i)_{i=1}^s$ by computing $y'_{n+1}$ at a local accuracy order of $s+2$. This can be done by replacing equation \eqref{eq2.2} by the equation
\begin{equation}
\label{jqx1}
y'_{n+1} = y'_n + h\tilde{d}_0f(t_n,y_n) + h\tilde{\bm{d}}^Tf(\bm{e}t_n+\bm{c}h,\bm{Y}_n), 
\end{equation}
where the coefficients $\tilde{d}_0$ and $\tilde{\bm{d}}^T$ are computed so that the local accuracy order of formula \eqref{jqx1} is $s+2$. We will see below that an alternative way to obtain this outcome is to choose  $(c_i)_{i=1}^s$  to satisfy the orthogonality condition in Theorem~\ref{theorem3.4} with $q=1$.
\end{rem}

\begin{rem}\rm \label{smooth2}Taking a Taylor expansion of $f(t,u(t))$ with respect to $u(t)$ at $y(t)$, we obtain
$$
L(t)=\sum_{k=1}^n\frac{1}{k!}\frac{\partial^k f}{\partial y^k}\big(t,y(t)\big)\big[u(t)-y(t)\big]^{k-1}+O\big([u(t)-y(t)]^{n}\big).
$$
This implies that $L(t)$ is quite smooth if $f(t,y)$ is sufficiently smooth. 
This leads to the existence of a Taylor expansion of $a_{ij}(t)$ and $b_{ij}(t)$, the entries of 
$\bm{X}(t)$ and $\bm{X}^{-1}(t)$. If $f(t,y)$ is sufficiently smooth, then so is the function $g$ defined in \eqref{eq3.1.0}. 
Thus, one has the following Taylor expansions
\begin{align}
\beta_{ij}(t):=b_{ij}(t)g(t)= & \beta_{ij}^{(0)}+\beta_{ij}^{(1)}t+\cdots+\beta_{ij}^{(s)}t^s+O(t^{s+1}),\qquad i,j=1,2. \label{eq3.18}
\end{align}
\end{rem}

\subsection{Superconvergence}

\begin{theorem}
\label{theorem3.4}
An $s$-stage FRKN method is of order $s+q$ if
the collocation parameters $(c_i)_{i = 1}^s$  satisfy
\begin{align}
\label{supper}
\int_{0}^1 \xi^j\prod_{i=1}^s(\xi-c_i)d\xi = 0,\quad j = 0, ..., q-1.
\end{align}
\end{theorem}

\begin{proof}
Let us give an alternative proof for Ozawa's result~\cite[Theorem~3]{O2} using our collocation framework.
From \eqref{eq3.16} and \eqref{eq3.18} one gets
\begin{equation}
\label{eq3.22}
\begin{split}
R'(h) &= h^{s+1}a_{21}(h)\int_0^1 b_{12}(\xi h)g(\xi h)\prod_{i=1}^s(\xi-c_i)\,d\xi\\
&\quad + h^{s+1}a_{22}(h)\int_0^1 b_{22}(\xi h)g(\xi h)\prod_{i=1}^s(\xi-c_i)\,d\xi\\
&=  h^{s+1}a_{21}(h)\sum_{i=0}^{s-1} h^i\beta_{12}^{(i)}\int_0^1  \xi^i\prod_{i=1}^s(\xi-c_i)d\xi\\
&\quad +  h^{s+1}a_{22}(h)\sum_{i=0}^{s-1} h^i \beta_{22}^{(i)}\int_0^1  \xi^i\prod_{i=1}^s(\xi-c_i)d\xi +O(h^{2s+1}).  
\end{split}
\end{equation}
It follows from \eqref{supper} and \eqref{eq3.22} that
\begin{equation}
\label{wr4}
y'_1 - y'(h) = R'(h) = O(h^{s+q+1}). 
\end{equation}

From \eqref{eq3.15} and Fubini's Theorem we have
\begin{equation}
\label{wr1}
\begin{split}
R(h) &= h^{s+2}\int_0^1 \int_0^t \bigg[a_{11}(th)\beta_{12}(\xi h)+ a_{12}(th)\beta_{22}(\xi h)\bigg]\prod_{i=1}^s(\xi-c_i)d\xi dt\\
&= h^{s+2}\int_0^1 \int_\xi^1 \bigg[a_{11}(th)\beta_{12}(\xi h) + a_{12}(th)\beta_{22}(\xi h)\bigg]\prod_{i=1}^s(\xi-c_i)\,dt d\xi\\
&= h^{s+2}\int_0^1 \beta_{12}(\xi h)\prod_{i=1}^s(\xi-c_i)\int_\xi^1 a_{11}(th)\,dt \,d\xi\\
&\quad + h^{s+2}\int_0^1 \beta_{22}(\xi h)\prod_{i=1}^s(\xi-c_i)\int_\xi^1 a_{12}(th)\,dt \,d\xi.
\end{split}
\end{equation}
We have
$$
h\int_\xi^1 a_{1i}(th)\,dt = \alpha_i(h) - \alpha_i(\xi h),\qquad \alpha_i(t) = \int_0^t a_{1i}(\xi)\, d\xi,\qquad i=1,2.
$$
This and equation \eqref{wr1} imply
\begin{equation}
\label{f23e1}
\begin{split}
R(h)&= h^{s+1}\int_0^1 [\alpha_{1}(h) - \alpha_1(\xi h)]\beta_{12}(\xi h)\prod_{i=1}^s(\xi-c_i)\,dt d\xi\\
&\quad+h^{s+1}\int_0^1  [\alpha_{2}(h) - \alpha_2(\xi h)]\beta_{22}(\xi h) \prod_{i=1}^s(\xi-c_i)\,dt d\xi.
\end{split}
\end{equation}
Let
\begin{equation}
\label{f23e2}
\alpha_i(t)\beta_{i2}(t) = \gamma_{i}^{(0)} + \gamma_{i}^{(1)} t + \cdots + \gamma_{i}^{(s)} t^s + O(t^{s+1}),\qquad i = 1,2.  
\end{equation}
Using equation \eqref{f23e1} and the Taylor expansions in \eqref{eq3.18} and \eqref{f23e2},  one gets
\begin{align}
\label{eq3.21}
R(h) &= h^{s+1} \sum_{i=0}^{s-1} h^i [\alpha_1(h)\beta_{12}^{(i)}  - \gamma_{1}^{(i)} ]\int_0^1  \xi^i\prod_{i=1}^s(\xi-c_i)d\xi  \\
&\quad+  h^{s+1} \sum_{i=0}^{s-1} h^i [\alpha_2(h)\beta_{22}^{(i)}  - \gamma_{2}^{(i)} ]\int_0^1  \xi^i\prod_{i=1}^s(\xi-c_i)d\xi  +O(h^{2s+1}). \nonumber
\end{align}
This and relation \eqref{supper} imply
\begin{equation}
\label{eqwe3}
y_1 - y(h) = R(h) = O(h^{s+q+1}),
\end{equation}
which proves Theorem~\ref{theorem3.4}.
\end{proof}

In particular, all $s$-stage FRKN methods based on Gauss points attain the maximum order of accuracy of $2s$
given that Gauss points satisfy the orthogonality condition~\eqref{supper} with $q=s$.


\begin{rem}{\rm One can verify that the relation
\begin{equation}
\label{wr2}
\alpha_i(h)\beta_{i2}^{(k)} - \gamma_i^{(k)} = O(h),\qquad i=1,2,
\end{equation}
holds for $k=0$. 
Thus, from equation \eqref{eq3.21} we still have $R(h) = O(h^{s+2})$ for any set of collocation parameters $(c_i)_{i=1}^s$. 
From \eqref{eqEp} and \eqref{eqE} one may ask whether the local order of accuracy  of $y_{n+1}$ is still one unit higher than the local order of accuracy of $y'_{n+1}$ under the orthogonality condition \eqref{supper}. The answer is negative because equation \eqref{wr2} does not hold for $k\ge 1$, in general. 

Equation \eqref{eqwe3} can also be obtained in a much shorter way as it is done for \eqref{wr4} by using the following equation 
\begin{align*}
R(t)=& \int_0^t [a_{11}(t)b_{12}(\xi)+ a_{12}(t)b_{22}(\xi)]g(\xi)\prod_{i=1}^s(\xi-c_ih)\,d\xi,
\end{align*}
which follows from \eqref{eqRode} and \eqref{eq3.14}. However, we chose the current approach in order to show that the local accuracy of $y_{n+1}$  cannot exceed the one of $y'_{n+1}$ under the orthogonality condition~\eqref{supper}. 
}
\end{rem}

\section{Linear stability}
\label{sec:Stability}

In~\cite{O2}, Ozawa only studied the order of accuracy of FRKN methods.
Here we present new results related to their stability that have not been reported so far in the literature.
We apply an $s$-stage FRKN method to the test problem
\begin{equation}
\label{eq4.1}
y'' = \lambda y,\quad y'(0) = y'_0,\quad y(0) = y_0,
\end{equation}
to get
\begin{equation}
\begin{pmatrix}
y_{n+1}\\hy'_{n+1} 
\end{pmatrix}
= M(z,h)
\begin{pmatrix}
y_{n}\\hy'_{n} 
\end{pmatrix},\qquad z := \lambda h^2,\qquad n=0,1,\ldots,
\end{equation}
where the matrix $M(z,h)$ is called the stability matrix of the corresponding FRKN method. 


Let us recall the following definitions (van der Houwen, Sommeijer and Cong~\cite{cong91}).
\begin{defn} For a given $h$,
the collection of points on the negative real $z$-axis is called 
\begin{itemize}
\item{the stability region if in this region $R_h(z) = \rho(M(z,h))<1$},
\item{the periodicity region if in this region $R_h(z)=1$ and $[\trace M(z,h)]^2 - 4 \det M(z,h) < 0$}.
\end{itemize}
\end{defn}

\begin{theorem}
\label{theorem4.1}
For a separable $s$-stage FRKN method characterized by the
set of basis functions $\bm{u}(t)=(1,t,u_1(t), \ldots, u_s(t))^T=e^{\bm{S}t}\bm{u}_0$,
the stability matrix is
\begin{equation}
\label{eq4.3}
\begin{split}
M(z,h) & =
\begin{pmatrix}
1 + z\bm{b}^T(h)[I- z\bm{A}(h)]^{-1}\bm{e} & 1 + z\bm{b}^T(h)[I-z\bm{A}(h)]^{-1}\bm{c}\\
z\bm{d}^T(h)[I- z\bm{A}(h)]^{-1}\bm{e} & 1 + z\bm{d}^T(h)[I-z\bm{A}]^{-1}\bm{c}
\end{pmatrix}\\
& = \begin{pmatrix}
\bm{e}_1^T \bm{W}^{-1}e^{\bm{S}h}\bm{u}_0 & h^{-1}\bm{e}_2^T \bm{W}^{-1} e^{\bm{S}h}\bm{u}_0\\
h\bm{e}_1^T \bm{W}^{-1} \bm{S}e^{\bm{S}h}\bm{u}_0 & \bm{e}_2^T \bm{W}^{-1} \bm{S}e^{\bm{S}h}\bm{u}_0
\end{pmatrix},
\end{split}
\end{equation}
where $\bm{e}_1$ and $\bm{e}_2$ denote the first and second columns
of the identity matrix of dimension $s+2$,  
and 
%
\begin{equation}
\label{eqR}
\begin{split}
\bm{W} &=
\begin{bmatrix}
\bm{u}_0 &\bm{S}\bm{u}_0& [(h\bm{S})^2 - z I]e^{\bm{S}c_1h}\bm{u}_0&\cdots 
&  [(h\bm{S})^2 - z I]e^{\bm{S}c_sh}\bm{u}_0
\end{bmatrix}
\end{split}
\end{equation}

\end{theorem}

\begin{proof}
Applying an $s$-stage FRKN method to the test 
equation \eqref{eq4.1} one gets the first equality in \eqref{eq4.3} (see, e.g.,~\cite{cong91}),
and this gives the stability matrix in terms of the coefficients of the FRKN method. 
To derive the second formula
that gives the stability matrix in terms of the matrix $\bm{S}$ that characterizes 
the basis functions, let $u(t)$ be the collocation solution to~\eqref{eq4.1}. 
Since it is a linear combination of the component functions of $\bm{u}(t)$, there exists 
$\bm{\mu}\in\mathbb{R}^{s+2}$ such that
\begin{equation}
\label{wr5}
u(t) = \bm{\mu}^T \bm{u}(t) = \bm{\mu}^T e^{\bm{S}t}\bm{u}_0.
\end{equation}
This implies 
\begin{equation}
\label{eq4.4}
P(t):=u'' - \lambda u =\bm{\mu}^T (\bm{S}^2 - \lambda I) e^{\bm{S}t}\bm{u}_0.
\end{equation}
Now $P(t)$ 
satisfies the collocation conditions $P(c_ih) = 0$, $i=1,...,s$. These identities together
with equation \eqref{eq4.4} and the initial conditions $u(0)=y_0$, $u'(0) = y'_0$ imply 
\begin{gather}
\bm{\mu}^T \bm{u}_0 = y_0, \quad
\bm{\mu}^T \bm{S}\bm{u}_0 = y'_0, \label{eqm4.5}\\
h^2P(c_ih) = \bm{\mu}^T [(h\bm{S})^2 - z I] e^{\bm{S}hc_i}\bm{u}_0 = 0,\quad i=1,\ldots,s,\qquad z:=h^2\lambda.
\label{eqm4.6}
\end{gather}
Equations \eqref{eqm4.5} and \eqref{eqm4.6} can be written together in matrix form as
\begin{equation*}
\bm{\mu}^T \bm{W} = 
\begin{pmatrix}
y_0   &y'_0   &0&\cdots &0
\end{pmatrix},
\end{equation*}
where $\bm{W}$ is defined as indicated in \eqref{eqR}.
Consequently, we get
\begin{equation*}
\bm{\mu}^T = 
\begin{pmatrix}
y_0   &y'_0   &0&\cdots &0
\end{pmatrix} \bm{W}^{-1} = y_0\bm{e}_1^T \bm{W}^{-1} + y'_0 \bm{e}_2^T \bm{W}^{-1}.
\end{equation*}
This and equation \eqref{wr5} imply
\begin{equation*}
u(t) = y_0 \bm{e}_1^T \bm{W}^{-1} e^{\bm{S}t}\bm{u}_0 + 
y'_0 \bm{e}_2^T \bm{W}^{-1} e^{\bm{S}t}\bm{u}_0.
\end{equation*}
Thus,
\begin{equation*}
u'(t) = y_0 \bm{e}_1^T \bm{W}^{-1} \bm{S}e^{\bm{S}t}\bm{u}_0 + 
y'_0 \bm{e}_2^T \bm{W}^{-1} \bm{S}e^{\bm{S}t}\bm{u}_0.
\end{equation*}
Therefore,
\begin{equation*}
\begin{pmatrix}
y_1 \\ hy'_1
\end{pmatrix}
= 
\begin{pmatrix}
u(h) \\ hu'(h)
\end{pmatrix}
= 
\begin{pmatrix}
\bm{e}_1^T \bm{W}^{-1} e^{\bm{S}h}\bm{u}_0 & h^{-1}\bm{e}_2^T \bm{W}^{-1} e^{\bm{S}h}\bm{u}_0\\
h\bm{e}_1^T \bm{W}^{-1} \bm{S}e^{\bm{S}h}\bm{u}_0 & \bm{e}_2^T \bm{W}^{-1} \bm{S}e^{\bm{S}h}\bm{u}_0
\end{pmatrix}
\begin{pmatrix}
y_0 \\ hy'_0
\end{pmatrix},
\end{equation*}
which implies the second inequality in \eqref{eq4.3}. 
\end{proof}

\section{An example}
\label{sec:Examples}

Consider the usual two-body gravitational problem with eccentricity $0 \le e < 1$,
\begin{equation}
\label{eq6.1}
\begin{split}
y_1'' &= - \frac{y_1}{(y_1 + y_2)^\frac{3}{2}},\quad y_1(0) = 1 - e,\quad y_1'(0) = 0,\\
y_2'' &= - \frac{y_2}{(y_1 + y_2)^\frac{3}{2}},\quad y_2(0) = 0,\quad y_2'(0) = \sqrt{\frac{1+e}{1-e}}.
\end{split}
\end{equation}

The solution to this problem is known to be
\begin{equation}
\label{eq6.2}
y_1(t) = \cos(u) - e,\qquad y_2(t) = \sqrt{1 - e^2} \sin(u),
\end{equation}
where $u$ is the solution of Keppler's equation $u = t + e\sin(u)$. 
We will use $ [0,20]$ as the integration domain of equation \eqref{eq6.1} in the numerical experiments. 

\subsection{Derivation of a 2-stage FRKN method}
\label{sec6.1}

We develop a 2-stage FRKN method to solve equation \eqref{eq6.1}
using the basis of functions $\{\cos(\omega t),\,\sin(\omega t)\}$
and Gauss points $(c_1,\,c_2) = 
\left(\frac{1}{2}-\frac{\sqrt{3}}{6},\,\frac{1}{2}+\frac{\sqrt{3}}{6}\right)$. 
The obtained method is denoted by FRKN2G and the corresponding classical collocation 2-stage method is denoted by RKN2G.  The suffix G is an indicator
that both methods use Gauss points.
From~\eqref{fFRKN} at $t=0$, the coefficients satisfy the system
\begin{subequations}
\begin{align*}
\sin(\bm{c}\nu)  &= \bm{c}\nu - \nu^2 \bm{A}(\nu)\sin(\bm{c}\nu),\qquad \nu = h\omega,\\
\cos(\bm{c}\nu) &= \bm{e} - \nu^2 \bm{A}(\nu)\cos(\bm{c}\nu),\\
\sin \nu &= \nu - \nu^2 \bm{b}^T(\nu) \sin(\bm{c}\nu),\\
\cos \nu &= 1 - \nu^2 \bm{b}^T(\nu) \cos(\bm{c}\nu),\\
\sin \nu &= \nu \bm{d}^T(\nu) \cos(\bm{c}\nu),\\
\cos \nu &= 1 - \nu \bm{d}^T(\nu) \sin(\bm{c}\nu).
\end{align*}
\end{subequations}
These equations can be written as follows
\begin{subequations}
\label{eq6.4}
\begin{align}
\begin{pmatrix}
\bm{c}\nu - \sin(\bm{c}\nu) & \bm{e} - \cos(\bm{c}\nu)
\end{pmatrix}
&= \nu^2 \bm{A}(\nu) \begin{pmatrix} \sin(\bm{c}\nu) & \cos(\bm{c}\nu)\end{pmatrix},\\
\begin{pmatrix}
\nu - \sin\nu  & 1 - \cos\nu
\end{pmatrix}
&= \nu^2 \bm{b}^T(\nu) \begin{pmatrix} \sin(\bm{c}\nu) & \cos(\bm{c}\nu)\end{pmatrix},\\
\begin{pmatrix}
 1 - \cos\nu & \sin\nu
\end{pmatrix}
&= \nu \bm{d}^T(\nu) \begin{pmatrix} \sin(\bm{c}\nu) & \cos(\bm{c}\nu)\end{pmatrix}. 
\end{align}
\end{subequations}
Solving these equations, one obtains the coefficients $\bm{A}(\nu), \bm{b}^T(\nu), \bm{d}^T(\nu)$
that are presented in the following Butcher tableau
$$
\renewcommand{\arraystretch}{1.2}
\begin{array}{c|cc}
c_1 &\displaystyle \frac{c_1\nu\cos(c_2\nu) - \sin(c_2\nu) -\sin((c_1-c_2)\nu)}{\nu^2\sin((c_1-c_2)\nu)}  &\displaystyle \frac{\sin(c_1\nu) - c_1\nu\cos(c_1\nu)}{\nu^2\sin((c_1-c_2)\nu)} \\ 
c_1 & \displaystyle\frac{c_2\nu\cos(c_2\nu) - \sin(c_2\nu)}{\nu^2\sin((c_1-c_2)\nu)} & \displaystyle \frac{\sin(c_1\nu)-c_2\nu\cos(c_1\nu) + \sin((c_2 - c_1)\nu)}{\nu^2\sin((c_1-c_2)\nu)} \\ 
\hline
      &\displaystyle \frac{\nu\cos(c_2\nu) - \sin(c_2\nu) -\sin((1-c_2)\nu)}{\nu^2\sin((1-c_2)\nu)}   & 
      \displaystyle \frac{\sin(c_1\nu)-\nu\cos(c_1\nu) + \sin((1 - c_1)\nu)}{\nu^2\sin((c_1-1)\nu)}\\
	&\displaystyle\frac{\cos(c_2\nu) - \cos((1-c_2)\nu)}{\nu\sin((c_1-c_2)\nu)} & \displaystyle\frac{\cos((1-c_1)\nu) - \cos(c_1\nu)}{\nu\sin((c_1-c_2)\nu)}
\end{array}
$$ 

We computed the coefficients more reliably  
by numerically solving the linear systems in equation~\eqref{eq6.4}
using Gaussian elimination with pivoting than using the 
closed-form representation in the tableau above (other authors
have successfully used Taylor expansions as well).
The trigonometric basis provides a reference of interest for our study.
It was used  in Ozawa~\cite{O2}, and it is often seen in the literature
as recently as in D'Ambrosio, Ferro, and Paternoster~\cite{D?Ambrosio}
who used it to construct a slightly different two-step method. 
Since the coefficients in the Butcher tableau above are functions of $\nu=\omega h$, the stability matrix $M(z,h)$ in equation 
\eqref{eq4.3} also depends on $\nu$. 
Using equation \eqref{eq4.3} we compute the spectral radius of the matrix $M(z,h)$ which we denote by $\rho(z)$, $z=\lambda h^2$, for various values of $\nu=\omega h$.

Figure \ref{figure1} plots the spectral radius of the stability matrix $M(z,h)$ of the FRKN2G method for various values of $\nu=\omega h$. 
\begin{figure}[!h!t!b!]
\centerline{
\begin{tabular}{c}
\mbox{\includegraphics[scale=0.85]{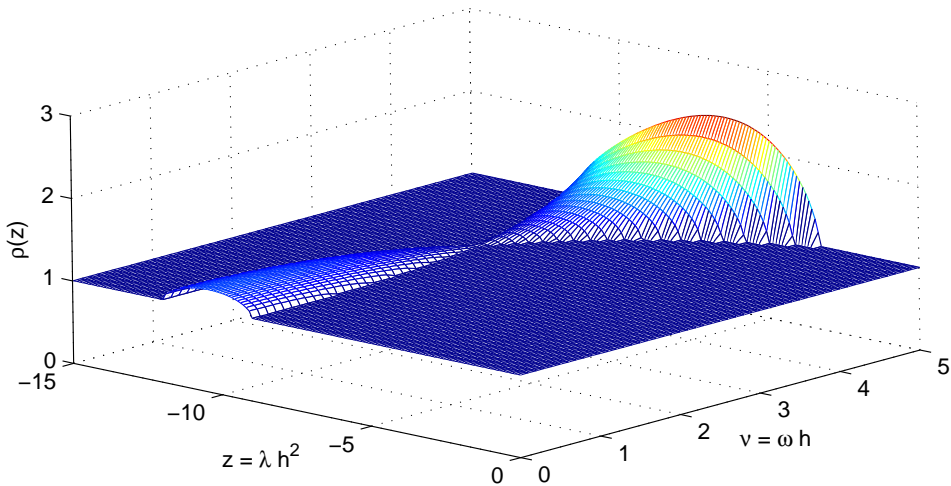}}
\end{tabular}
}
\caption{\it Spectral radius $\rho(z)$ of the stability matrix $M_\nu(z)$ of the FRKN2G method.
}
\label{figure1}
\end{figure}

Figure \ref{figure2} plots the stability region of  the FRKN2G method in the $z$ and $\nu=\omega h$ plane. 
From Figure \ref{figure2} we can see that the stability region of  FRKN2G  increases as $\nu=\omega h$ increases from $0$ to $\pi$, and contains the interval $[-9,0]$ for all $\nu\in [0,\pi]$. The stability region decreases as $\nu$ increases from $\pi$ to 5.4. 
When $\nu$ is in the interval $[5.5,2\pi]$, FRKN2G  is unstable for any small $z$, or equivalently, for any small stepsize $h$. This suggests not to use  FRKN2G  with $\nu\in [5.5,2\pi]$. Moreover, when using  FRKN2G  one should restrict $\nu$ to be in the interval $[0,\pi]$ to have  large stability regions. It is known that the coefficients of  FRKN methods converge to those of the corresponding classical RKN methods. Thus, the limit of the stability region of  FRKN2G  as $\nu$ tends to zero is the same as the stability region of  RKN2G. From Figure \ref{figure2} one can see that the stability region of  FRKN2G  is slightly larger than that of  RKN2G  when $\nu\in [0,\pi]$. 
\begin{figure}[!h!t!b!]
\centerline{
\begin{tabular}{c}
\mbox{\includegraphics[scale=0.75]{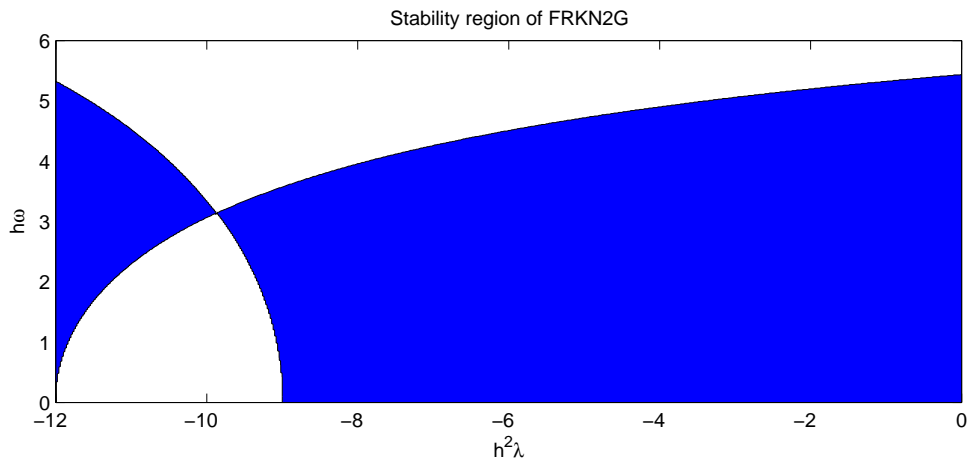}}
\end{tabular}
}
\caption{\it Plots of stability region (shaded) of the FRKN2G method for various value of $\nu=\omega h$.
}
\label{figure2}
\end{figure}

\subsection{Numerical experiments}


In this section we implement the 2-stage FRKN2G method derived in section~\ref{sec6.1},
and then compare it with the RKN2G method. 

In the first experiment we solve equation \eqref{eq6.1} by FRKN2G and RKN2G  for $e=0.5$. 
Table~\ref{table1} shows the numerical results obtained by the two methods.
It is clear from equation \eqref{eq6.2} that if $e$ is large, then the solution of equation \eqref{eq6.1} is not well approximated by a linear combination of $\cos t$ and $\sin t$. Thus, for large values of $e$, we do not expect  FRKN2G  to be better than  RKN2G.  From Table \ref{table1} one can see that  RKN2G is slightly better than  FRKN2G. However, the difference between the two methods is insignificant. 
\begin{table}[ht] 
\caption{Errors $\triangle y_i = \log_{10}\displaystyle\max_{t_0\le t_n\le t_N}
|y_i(t_n)-y_{i,\text{comput}}(t_n)|$ over $N$ integration steps
for the two components $i=1,2$ when $e=0.5$.}
\label{table1}
\centering
\small
\renewcommand{\arraystretch}{1.2}
\begin{tabular}{@{}l@{}c|@{\hspace{2mm}}c@{\hspace{2mm}}c@{\hspace{2mm}}|c@{\hspace{2mm}}c@{\hspace{2mm}}|
c@{\hspace{2mm}}c@{\hspace{2mm}}|}
\cline{2-6}
&	&\multicolumn{2}{c|}{FRKN2G} & 
\multicolumn{2}{c|}{RKN2G} \\
&$h$	&$\triangle y_1$&$\triangle y_2$&$\triangle y_1$&$\triangle y_2$ \\
\hline
    &$1/2^{~}$   &-0.1555   &-0.0703   &-0.0643   &-0.0009\\
    &$1/2^2$   &-1.4358   &-1.2576   &-1.4889   &-1.3038 \\
    &$1/2^3$   &-3.0069   &-2.7745   &-3.1459   &-2.8956 \\
    &$1/2^4$  &-4.1495   &-3.9321   &-4.2650   &-4.0354  \\
    &$1/2^5$   &-5.3323   &-5.1172   &-5.4399   &-5.2148 \\
    &$1/2^6$   &-6.5308   &-6.3167   &-6.6365   &-6.4128 \\
    &$1/2^7$   &-7.7340   &-7.5201   &-7.8388   &-7.6154 \\
    &$1/2^8$   &-8.9457   &-8.7315   &-9.0424   &-8.8192 \\
\hline
\end{tabular} 
\end{table}

Table \ref{table2} presents the numerical result when $e=0.01$. It follows from Table \ref{table2} that   FRKN2G is much better than the classical RKN2G method. This follows from the fact that when $e$ is small the solution to equation \eqref{eq6.1} can be well approximated by a linear combination of $\sin t$ and $\cos t$. Indeed, in this experiment,  RKN2G  requires a stepsize half of the one used by  FRKN2G to yield results of the same accuracy. 
Consequently, in this experiment,  RKN2G  takes twice as long as  FRKN2G for the same accuracy.
\begin{table}[ht] 
\caption{Errors $\triangle y_i = \log_{10}\displaystyle\max_{t_0\le t_n\le t_N}
|y_i(t_n)-y_{i,\text{comput}}(t_n)|$ over $N$ integration steps
for the two components $i=1,2$ when $e=0.01$.}
\label{table2}
\centering
\small
\renewcommand{\arraystretch}{1.2}
\begin{tabular}{@{}l@{}c|@{\hspace{2mm}}c@{\hspace{2mm}}c@{\hspace{2mm}}|c@{\hspace{2mm}}c@{\hspace{2mm}}|
c@{\hspace{2mm}}c@{\hspace{2mm}}|}
\cline{2-6}
&	&\multicolumn{2}{c|}{FRKN2G} & 
\multicolumn{2}{c|}{RKN2G} \\
&$h$	&$\triangle y_1$&$\triangle y_2$&$\triangle y_1$&$\triangle y_2$ \\
\hline
    &$1/2$   &-4.0500   &-3.7300   &-2.3942   &-2.4200   \\
    &$1/2^2$   &-5.1726   &-4.8342   &-3.5973   &-3.5971  \\
    &$1/2^3$   &-6.3231   &-6.0228   &-4.8289   &-4.8213 \\
    &$1/2^4$   &-7.5164   &-7.2231   &-6.0429   &-6.0354 \\
    &$1/2^5$   &-8.7176   &-8.4263   &-7.2502   &-7.2426 \\
    &$1/2^6$   &-9.9273   &-9.6343   &-8.4551   &-8.4475 \\
    &$1/2^7$  &-11.5489  &-11.1156   &-9.6596   &-9.6519 \\
\hline
\end{tabular} 
\end{table}  

We performed further experiments not reported here with smaller values of $e$ and observed that
the smaller $e$ is, the better FRKN2G  compared to RKN2G. 

We also carried out  experiments with a different 2-stage FRKN method constructed with coefficients found from equation 
\eqref{eq6.4} using $(c_1,c_2) = (0.2,1)$. This method is denoted by FRKN2
without the suffix G that was used earlier to indicate Gauss points. The corresponding 2-stage classical RKN method using 
$(c_1,c_2) = (0.2,1)$ is also simply denoted by RKN2.
To improve the accuracy for computing $y'_{n+1}$ we derived a further method using the formula
\begin{equation}
\label{jqx21}
y'_{n+1} = y'_n + h\tilde{d}_0(\nu)f(t_n,y_n) + h\tilde{\bm{d}}^T(\nu)f(\bm{e}t_n + \bm{c}h,\bm{Y}_n),
\end{equation}
where the coefficients $\tilde{d}_0(\nu)$ and $\tilde{\bm{d}}^T(\nu)$ are computed so that the formula \eqref{jqx21} has 
a local accuracy order of $s+2=4$. The resulting method with $y'_{n+1}$ computed by \eqref{jqx21} 
is denoted by FRKN2x. Concomitantly, we use RKN2x to denote the corresponding
classical method with $y'_{n+1}$ computed by a similar equation to \eqref{jqx21}. 

Table \ref{table3} presents the numerical results for FRKN2, FRKN2x, RKN2, and RKN2x,  for the two-body problem with 
$e=0.5$. It follows from the table that the accuracy orders for FRKN2 and RKN2 are both $p=2$ while the accuracy orders for FRKN2x and RKN2x are both $p=3$. This confirms our observation earlier that it is because of the accuracy order of $y'_{n+1}$ that the accuracy of the method cannot exceed $s$ for arbitrary $\bm{c}=(c_i)_{i=1}^s$. From this table one can see that the FRKN2 and FRKN2x methods are comparable to the corresponding classical RKN2 and RKN2x methods.  

\begin{table}[ht] 
\caption{Errors $\triangle y_i = \log_{10}\displaystyle\max_{t_0\le t_n\le t_N}
|y_i(t_n)-y_{i,\text{comput}}(t_n)|$ over $N$ integration steps
for the two components $i=1,2$ when $e=0.5$.}
\label{table3}
\centering
\small
\renewcommand{\arraystretch}{1.2}
\begin{tabular}{@{}l@{}c|@{\hspace{2mm}}c@{\hspace{2mm}}c@{\hspace{2mm}}|c@{\hspace{2mm}}c@{\hspace{2mm}}|
c@{\hspace{2mm}}c@{\hspace{2mm}}|c@{\hspace{2mm}}c@{\hspace{2mm}}|}
\cline{2-10}
&	&\multicolumn{2}{c|}{FRKN2} & 
\multicolumn{2}{c|}{FRKN2x} &\multicolumn{2}{c|}{RKN2} & 
\multicolumn{2}{c|}{RKN2x}\\
&$h$	&$\triangle y_1$&$\triangle y_2$&$\triangle y_1$&$\triangle y_2$ 
&$\triangle y_1$&$\triangle y_2$&$\triangle y_1$&$\triangle y_2$\\ 
\hline    
    &$1/2^4$    &-0.6175   &-0.4361   &-1.3312   &-1.1528   &-0.5945   &-0.4147   &-1.3046   &-1.1290\\
    &$1/2^5$    &-1.2154   &-1.0278   &-2.2383   &-2.0605   &-1.1917   &-1.0048   &-2.2152   &-2.0402\\
    &$1/2^6$    &-1.8149   &-1.6267   &-3.1432   &-2.9656   &-1.7909   &-1.6034   &-3.1217   &-2.9470\\
    &$1/2^7$    &-2.4154   &-2.2272   &-4.0471   &-3.8696   &-2.3912   &-2.2037   &-4.0265   &-3.8519\\
    &$1/2^8$    &-3.0166   &-2.8284   &-4.9506   &-4.7732   &-2.9924   &-2.8049   &-4.9305   &-4.7559\\
    &$1/2^9$    &-3.6182   &-3.4300   &-5.8540   &-5.6765   &-3.5939   &-3.4064   &-5.8340   &-5.6595\\
    &$1/2^{10}$    &-4.2201   &-4.0318   &-6.7573   &-6.5799   &-4.1957   &-4.0083   &-6.7373   &-6.5628\\
    &$1/2^{11}$    &-4.8220   &-4.6338   &-7.6647   &-7.4872   &-4.7977   &-4.6102   &-7.6405   &-7.4660\\
\hline
\end{tabular} 
\end{table}  

Table \ref{table4} presents the numerical results for the four methods FRKN2, FRKN2x, RKN2, and RKN2x when $e=0.01$. 
Again, it follows from Table \ref{table4} that FRKN2x and RKN2x have an accuracy order $p=3$ while the 
FRKN2 and RKN2 have an accuracy order $p=2$. One can see from Table \ref{table4} that the FRKN2 and FRKN2x methods are much better than the RKN2 and RKN2x methods because, as we have discussed earlier, when $e$ is small, the exact solution to the two-body problem can be well approximated by the basis $\{\cos(t),\sin(t)\}$ used to construct FRKN2 and FRKN2x. 

\begin{table}[ht] 
\caption{Errors $\triangle y_i = \log_{10}\displaystyle\max_{t_0\le t_n\le t_N}
|y_i(t_n)-y_{i,\text{comput}}(t_n)|$ over $N$ integration steps
for the two components $i=1,2$ when $e=0.01$.}
\label{table4}
\centering
\small
\renewcommand{\arraystretch}{1.2}
\begin{tabular}{@{}l@{}c|@{\hspace{2mm}}c@{\hspace{2mm}}c@{\hspace{2mm}}|c@{\hspace{2mm}}c@{\hspace{2mm}}|
c@{\hspace{2mm}}c@{\hspace{2mm}}|c@{\hspace{2mm}}c@{\hspace{2mm}}|}
\cline{2-10}
&	&\multicolumn{2}{c|}{FRKN2} & 
\multicolumn{2}{c|}{FRKN2x} &\multicolumn{2}{c|}{RKN2} & 
\multicolumn{2}{c|}{RKN2x}\\
&$h$	&$\triangle y_1$&$\triangle y_2$&$\triangle y_1$&$\triangle y_2$ 
&$\triangle y_1$&$\triangle y_2$&$\triangle y_1$&$\triangle y_2$\\ 
\hline    
    &$1/2^3$    &-2.7401   &-2.6147   &-3.8219   &-3.9469   &-1.7383   &-1.7175   &-1.7393   &-1.7567\\
    &$1/2^4$    &-3.3446   &-3.2180   &-4.7298   &-4.8702   &-2.3078   &-2.2835   &-2.6401   &-2.6591\\
    &$1/2^5$    &-3.9454   &-3.8201   &-5.6354   &-5.7843   &-2.8940   &-2.8680   &-3.5427   &-3.5620\\
    &$1/2^6$    &-4.5469   &-4.4222   &-6.5398   &-6.6932   &-3.4884   &-3.4614   &-4.4457   &-4.4649\\
    &$1/2^7$    &-5.1486   &-5.0242   &-7.4437   &-7.5993   &-4.0866   &-4.0592   &-5.3487   &-5.3679\\
    &$1/2^8$    &-5.7505   &-5.6263   &-8.3477   &-8.5049   &-4.6868   &-4.6592   &-6.2517   &-6.2710\\
    &$1/2^9$    &-6.3525   &-6.2283   &-9.2636   &-9.4296   &-5.2879   &-5.2602   &-7.1548   &-7.1741\\
    &$1/2^{10}$    &-6.9547   &-6.8305  &-10.4191  &-10.8514   &-5.8895   &-5.8617   &-8.0579   &-8.0772\\
\hline
\end{tabular} 
\end{table}  

\section{Conclusion.}
\label{sec:Conclusion}

We studied functionally fitted Runge-Kutta-Nystr\"{o}m  methods using the collocation framework  that we previously introduced for functionally fitted Runge-Kutta methods. This study, therefore, fills a gap
in the literature by unifying the framework for both families of methods. 
We recovered earlier results of Ozawa~\cite{O2}
regarding the order of accuracy and the superconvergence, 
and  established new stability results. 
Numerical experiments showed that FRKN methods tuned
for a particular problem could indeed perform better than general purpose methods.




\begin{thebibliography}{99}

\bibitem{Coleman}
J. P. Coleman and S.C. Duxbury, 
 Mixed collocation methods for $y''=f(t,y)$,
 {\em J. Comput. Appl. Math.}, 126 (2000), 47--75.

\bibitem{D?Ambrosio}
R. D'Ambrosio, M. Ferro, B. Paternoster, 
Trigonometrically fitted two-step hybrid methods for special second order ordinary differential equations,
 {\em Math. Comput. Simulat.},  81 (2011), no. 5, 1068--1084, 


\bibitem{Franco2}
J. M. Franco,
 Exponentially fitted explicit {R}unge-{K}utta-{N}ystr\"{o}m methods, 
 {\em J. Comput. Appl. Math.}, 167 (2004), 1--19.

\bibitem{Franco&Gomez2013}
J. M. Franco and I. G{\'o}mez, 
Some procedures for the construction of high-order exponentially fitted 
{R}unge-{K}utta-{N}ystr\"{o}m methods of explicit type,
 {\em Comput. Phys. Commun.}, 184 (2013), 1310--1321.


\bibitem{FRK}
N. S. Hoang, R. B. Sidje, N. H. Cong,
On functionally-fitted {R}unge-{K}utta methods, {\em BIT Numer. Math.}, 46 (2006), no. 4, 861--874. 

\bibitem{SFRK}
N. S. Hoang, R. B. Sidje, On the stability of functionally fitted Runge-Kutta methods, {\em BIT Numer. Math.}, 48 (2008), no. 1, 61--77.


\bibitem{O2}
K.~Ozawa, 
 Functional fitting {R}unge-{K}utta-{N}ystr\"{o}m method with variable
  coefficients,
 {\em Japan J. Indust. App. Math.}, 19 (2002), 55--85.


\bibitem{Pater}
B.~Paternoster, 
 {R}unge-{K}utta(-{N}ystr\"{o}m) methods for {ODEs} with periodic
  solutions based on trigonometric polynomials,
 {\em Appl. Num. Math.}, 28 (1998), 401--412.

\bibitem{cong91}
P. J. van der Houwen, B. P. Sommeijer, N. H. Cong, Stability of collocation-based Runge-Kutta-Nystr\"{o}m methods, {\em BIT Numer. Math.}, 31 (1991), no. 3, 469--481.



\end{thebibliography}
\end{document}